\documentclass[10pt,a4paper,twoside,reqno,nonamelimits]{amsart}
\usepackage[colorlinks=true,urlcolor=blue,linkcolor=blue,citecolor=blue]{hyperref}
\usepackage{mathptmx}
\usepackage{amsmath}
\usepackage{amssymb}
\usepackage{amsfonts}
\usepackage{amscd}
\usepackage{amsthm}
\usepackage{amsxtra}
\usepackage[all]{xy}  \CompileMatrices
\usepackage{mathrsfs}
\usepackage{nicefrac}
\usepackage{bbold}
\usepackage{graphicx}
\usepackage{color}
\usepackage{enumerate}
\usepackage{wrapfig}

\theoremstyle{plain}
\newtheorem{theorem}{Theorem}

\newtheorem*{proposition*}{Proposition}
\newtheorem{corollary}[theorem]{Corollary}
\newtheorem*{corollary*}{Corollary}
\newtheorem{lemma}[theorem]{Lemma}

\newtheorem*{theorem*}{Theorem}
\newtheorem*{lemma*}{Lemma}
\newtheorem*{conjecture*}{Conjecture}

\newtheorem*{question*}{Question}
\newtheorem{question}[theorem]{Question}
\newtheorem*{problem*}{Problem}

\theoremstyle{definition}

\newtheorem*{exercise*}{Exercise}

\theoremstyle{remark}
\newtheorem{remark}[theorem]{Remark}
\newtheorem*{remark*}{Remark}

\newtheorem*{remarks*}{Remarks}

\newtheorem*{claim*}{Claim}

\newcommand{\subclass}[1]{}

\newcommand{\enumTi}[1]{\renewcommand{\theenumi}{#1}}

\newcommand{\alphenumi}{\enumTi{\alph{enumi}}}

\newcommand{\romenumi}{\enumTi{\roman{enumi}}}

\renewcommand{\em}{\sl}

\DeclareMathOperator{\rk}{rk}

\newcommand{\lt}{\left}
\newcommand{\rt}{\right}

\newcommand{\abs}[1]{{\lt\lvert{#1}\rt\rvert}}

\newcommand{\FF}{\mathbb{F}}

\newcommand{\NN}{\mathbb{N}}

\newcommand{\QQ}{\mathbb{Q}}
\newcommand{\RR}{\mathbb{R}}

\newcommand{\ZZ}{\mathbb{Z}}
\newcommand{\kk}{\mathbb{k}}

\newlength{\algotabbingwidth}
\renewcommand{\paragraph}[1]{\par\smallskip\noindent{#1}}

\newcommand{\Mt}{M^{\scriptscriptstyle(t)}}
\newcommand{\nt}{n^{\scriptscriptstyle(t)}}
\newcommand{\Mr}[1][r]{M^{\scriptscriptstyle(#1)}}

\begin{document}

\title[Fooling-sets and rank]{Fooling-sets and rank}%


\newsavebox{\fna}\sbox{\fna}{{\footnotesize${}^{a}$}}
\newsavebox{\fnb}\sbox{\fnb}{{\footnotesize${}^{b}$}}
\newsavebox{\fnc}\sbox{\fnc}{{\footnotesize${}^{c}$}}
\newsavebox{\fnd}\sbox{\fnd}{{\footnotesize${}^{d}$}}
\author[Mirjam Friesen]{Mirjam Friesen\usebox{\fna}}%
\author[Aya Hamed]{Aya Hamed\usebox{\fnb}}
\author[Troy Lee]{Troy Lee\usebox{\fnc}}
\author[Dirk Oliver Theis]{Dirk Oliver Theis\usebox{\fnd}}%
\date{{Thu Jan 16 16:32:31 EET 2014}\\[1ex] %
  ${}^a$ Faculty of Mathematics, Otto von Guericke University Magdeburg, Germany\\
  ${}^b$ {Work done in part while visiting the Centre for Quantum Technologies, Singapore}\\
  ${}^c$ {Nanyang Technological University and Centre for Quantum Technologies, Singapore.  Research supported by a National Research Foundation Fellowship.}\\
  ${}^d$ Institute of Computer Science, University of Tartu, Estonia.  \tiny\texttt{dirk.oliver.theis@ut.ee}}

\begin{abstract}
  An $n\times n$ matrix $M$ is called a \textit{fooling-set matrix of size $n$} if its diagonal entries are nonzero and $M_{k,\ell} M_{\ell,k} = 0$ for every $k\ne \ell$.  Dietzfelbinger, Hromkovi{\v{c}}, and Schnitger (1996) showed that $n \le (\rk M)^2$, regardless of over which field the rank is computed, and asked whether the exponent on $\rk M$ can be improved.

  \noindent%
  We settle this question.
  In characteristic zero, we construct an infinite family of rational fooling-set matrices with size $n = \binom{\rk M+1}{2}$.
  In nonzero characteristic, we construct an infinite family of matrices with $n= (1+o(1))(\rk M)^2$.
\end{abstract}
\maketitle


\section{Introduction}
An $n\times n$ matrix~$M$ over a field~$\kk$ is called a \textit{fooling-set matrix of size~$n$} if
\begin{subequations}\label{eq:def-fool}
  \begin{align}
    M_{kk} &\ne 0 &&\text{ for all~$k$ (its diagonal entries are all nonzero), and} \label{eq:def-fool:diag}\\
    M_{k,\ell} \, M_{\ell,k} &= 0 &&\text{ for all $k\ne \ell$.}  \label{eq:def-fool:off-diag}
  \end{align}
\end{subequations}
Note that the definition depends only on the zero-nonzero pattern of~$M$.  The word ``fooling set'' originates from Communication Complexity, but the concept is used under different names in other contexts (see Section~\ref{sec:connect}).

In Communication Complexity and Combinatorial Optimization fooling-set matrices are used to show lower bounds on other numerical properties of interest.  To do this, one wants to find a large fooling-set (sub-)matrix contained in a given matrix~$A$, where permutation of rows and columns is allowed.  Since large fooling-set submatrices are typically difficult to identify (deciding whether a fooling-set submatrix of given size exists in a given matrix was recently shown to be NP-hard~\cite{Shitov13fool}), it is desirable to upper-bound the size of a fooling-set matrix one may possibly hope for in terms of easily computable properties of~$A$.

Dietzfelbinger, Hromkovi{\v{c}}, and Schnitger (\cite[Thm.~1.4]{DietzfelbingerHromkovicSchnitger96}, or see~\cite[Lemma~4.15]{KushilevitzNisan97}; cf.~\cite{KlauckDewolf13,FioriniKaibelPashkovichTheis13}) proved that the rank of a fooling-set matrix of size~$n$ is at least $\sqrt n$, i.e.,
\begin{equation}\label{eq:rk-of-fool}
  n \le (\rk_{\kk} M)^2.
\end{equation}
This bound follows as $\rk_{\kk} I_n = \rk_{\kk} M \circ M^T \le (\rk_{\kk} M)^2$, where $I_n$ is the identity matrix of size $n$ and $\circ$ denotes entrywise product.  This inequality gives such an upper bound on the largest fooling-set submatrix in terms of the easily computable rank of~$A$.

Dietzfelbinger et al.\ asked the question whether the exponent on the rank in the right-hand side of~\eqref{eq:rk-of-fool} can be improved or not \cite[Open Problem~2]{DietzfelbingerHromkovicSchnitger96}.  This problem is stated specifically for 0/1-matrices in their paper, mirroring the particular Communication Complexity situation studied there.  Klauck and de Wolf~\cite{KlauckDewolf13}, however, gave applications and pointed out the importance for Communication Complexity of the question regarding general (i.e., not 0/1) matrices.  For applications in Combinatorial Optimization, 0/1 matrices play no special role.

Currently, the examples (attributed to M.~H\"uhne in~\cite{DietzfelbingerHromkovicSchnitger96}) of fooling-set matrices~$M$ with smallest rank are such that $n \approx (\rk_{\FF_2} M)^{\log_4 6}$ ($\log_4 6 = 1.292\dots$); for general matrices, Klauck and de Wolf~\cite{KlauckDewolf13} have given examples with $n \approx (\rk_\QQ M)^{\log_3 6}$ ($\log_3 6 = 1.63\dots$).

\paragraph{\bf In this paper, we settle this question.}
Firstly, for the case that $\kk$ has nonzero characteristic, we prove that the inequality~\eqref{eq:rk-of-fool} is asymptotically tight.  Notably, not only is the exponent on the rank in inequality~\eqref{eq:rk-of-fool} best possible, but so is the constant (one) in front of the rank.  We do this by constructing an infinite family fooling-set matrices~$M$ over $\kk=\FF_p$ of size~$n$, for with $n= (1+o(1))(\rk M)^2$.  The construction is based on a periodic sequence involving binomial coefficients.\footnote{%
  An extended abstract of this part of the current paper appeared in the EuroComb'13 proceedings~\cite{FriesenTheis13}.
} %

Secondly, in characteristic zero, we prove that the inequality is best possible up to a multiplicative constant, by constructing, for infinitely many~$n$, fooling-set matrices~$M$ over $\kk=\QQ$ of size~$n$, with $n = \binom{\rk M+1}{2}$.  This construction is inspired by the relations between binomial coefficients which used in the nonzero characteristic.

The method used in all the \emph{earlier} examples mentioned above of fooling-set matrices with small rank was the following: One conjures up a single, small fooling-set matrix~$M^0$ (of size, say, 6), determines its rank (say, 3), and then uses the tensor-powers of~$M^0$ (which are fooling-set matrices, too).  With these numerical values, from~$M^0$, one obtains $\log_36$ as a lower bound on the exponent on the rank in~\eqref{eq:rk-of-fool}.

Our constructions are departures from this approach.  In the characteristic $k>0$ case our matrices are circulant.  For the characteristic $k=0$ case, the matrices have a more complicated block structure, but each block is Toeplitz.

\paragraph{\bf Organization of this paper.}  %
In the next section we will explain some of the connections of the fooling-set vs.\ rank problem with Combinatorial Optimization and Graph Theory concepts.
In Section~\ref{sec:construction}, we prove our result for nonzero characteristic, and in Section~\ref{sec:chzero}, we prove the result for characteristic zero.

In the final section, we discuss some consequences and point to some questions which remain open.

\section{Some Remarks on the Importance of Fooling-Set Matrices}\label{sec:connect}
While the fooling-set size vs.\ rank problem is of interest in its own right as a minimum-rank type problem in Combinatorial Matrix Theory, fooling-set matrices are connected to other areas of Mathematics and Computer Science.

\paragraph{\bf In Polytope Theory,} %
given a polytope~$P$, sizes of fooling-set submatrices of appropriately defined matrices provide lower bounds to the number of facets of any polytope~$Q$ which can be mapped onto~$P$ by a projective mapping. We sketch the connection (see~\cite{FioriniKaibelPashkovichTheis13} for the details).

Let~$P$ be a polytope.  Let $A=A(P)$ be a matrix whose rows are indexed by the facets of~$P$ and whose columns are indexed by the vertices of~$P$, and which satisfies
$A_{F,v} = 0$, if $v \in F$, and $A_{F,v} \ne 0$, if $v \notin F$.
The following was first observed by Yannakakis (see~\cite{FioriniKaibelPashkovichTheis13} for a direct proof).
\begin{theorem}[\cite{Yannakakis91}]
  If~$A$ has a fooling-set submatrix of size~$n$, then every polytope~$Q$ which can be mapped onto~$P$ by a projective mapping has at least~$n$ facets.
\end{theorem}
Since for any fooling-set submatrix of size~$n$ of~$A$, the inequality
\begin{equation}\label{eq:dietz:ptp}
  n \le (\dim P+1)^2.
\end{equation}
follows from~\eqref{eq:rk-of-fool} (cf.~\cite{FioriniKaibelPashkovichTheis13}), the following variant of Dietzfelbinger et al.'s question is of pertinence in Polytope Theory: \textit{Can the fooling-set size vs.\ dimension inequality~\eqref{eq:dietz:ptp} be improved for polytopes?}  Our Theorem~\ref{thm:main:0} below yields the following corollary.

\begin{corollary}
  For infinitely many~$d$, there is a a polytope~$P$ of dimension~$d$ such that the matrix~$A(P)$ contains a fooling-set submatrix of size $\Omega(\sqrt d)$.
\end{corollary}
We do not prove this corollary in this paper, because it would require a considerable amount of polytope theory overhead to arrive at the a comparatively easy consequence of Theorem~\ref{thm:main:0}.  As a quick sketch, let the following suffice.  From a given matrix~$A$, one derives a pointed convex polyhedral cone by taking a rank factorization of~$A'$.  Intersecting the cone with a hyperplane gives the desired polytope~$P$.  The presence of rows/columns in~$A'$ which do not correspond to facets/vertices of~$P$ is not a problem by Proposition~5.4 in \cite{FioriniKaibelPashkovichTheis13}.

In \textbf{Combinatorial Optimization,} the polytope theoretic situation occurs for particular families of polytopes which arise from combinatorial optimization problems.  Sizes of fooling-set matrices then yield lower bounds to the minimum sizes of Linear Programs for combinatorial optimization problems~\cite{Yannakakis91}.
See~\cite{FioriniKaibelPashkovichTheis13} for bounds based on fooling sets for a number of combinatorial optimization problems, including bipartite matching.

In the Polytope Theory / Combinatorial Optimization applications, we typically have $\kk=\QQ$, and the rank of the large matrix~$A$ is known.  However, since the definition of a fooling-set matrix depends only on the zero-nonzero pattern, changing the field from $\QQ$ to $\kk'$ and replacing the nonzero rational entries of~$A$ by nonzero numbers in~$\kk'$ may yield a matrix with lower rank and hence a better upper bound on the size of a fooling-set matrix.

\paragraph{\bf In Computational Complexity,} %
fooling-set matrices provide lower bounds for the communication complexity of Boolean functions (see, e.g., \cite{AroraBarak09,KushilevitzNisan97,LovaszSaks88Moeb,DietzfelbingerHromkovicSchnitger96,KlauckDewolf13}), and for the number of states of an automaton accepting a given language (e.g., \cite{GruberHolzer06}).

As an example from Communication Complexity where the ``fooling-set method'' can be seen to yield a poor lower bound is the inner product function
\begin{equation*}
  f(x,y) = \sum_{j=1}^n x_jy_j,\qquad\text{for $x,y\in \ZZ_2^n$.}
\end{equation*}
The rank of the associated $2^n\times 2^n$-matrix is~$n$, hence, by~\eqref{eq:rk-of-fool}, there is no fooling-set submatrix larger than $n^2$.

\paragraph{\bf In Graph Theory,} %
a fooling-set matrix (up to permutation of rows and columns) can be understood as the incidence matrix of a bipartite graph containing a perfect cross-free matching.  Recall that a matching in a bipartite graph~$H$ is called \textit{cross-free} if no two matching edges induce a~$C_4$-subgraph of~$H$.

Cross-free matchings are best known as a lower bound on the size of biclique coverings of graphs (e.g.\ \cite{Dawande03,JuknaKulikov09}).  A \textit{biclique covering} of a graph~$G$ is a collection of complete bipartite subgraphs of~$G$ such that each edge of~$G$ is contained in at least one of these bipartite subgraphs.  If a cross-free matching of size~$n$ is contained as a subgraph in~$G$, then at least~$n$ bicliques are needed to cover all edges of~$G$.  For some classes of graphs, this is a sharp lower bound on the biclique covering number~\cite{Dawande03,SotoTelha11}.

\paragraph{\bf In Matrix Theory,} the maximum size of a fooling-set submatrix is known under a couple of different names, e.g.~as independence number \cite[Lemma 2.4]{CohenRothblum93}), or as intersection number.  For some semirings, this number provides a lower bound for the factorization rank of the matrix over the semiring.


\paragraph{\bf In each of these areas,} %
fooling-set matrices are used as lower bounds.  Upon embarking on a search for a big fooling-set matrix in a large, complicated matrix~$A$, one is interested in an \textit{a priori} upper bound on their sizes and thus the potential usefulness of the lower bound method.


\section{Preliminaries}
We will make use of binomial coefficients and a few of their standard properties.  As multiple extensions of binomial coefficients to negative arguments are possible, we fix here the definition we use (following~\cite{KnuthGrahamPatashnik94}).  For intgers $n,k$, let
\begin{equation*}
  \binom{n}{k} := \begin{cases}
    \dfrac{n(n-1) \cdots (n-k+1)}{k(k-1) \cdots 1}, &\text{if } k \ge 0, \\
    0, &\text{if } k <0.
  \end{cases}
\end{equation*}
Note that the \textit{symmetry identity}
\begin{equation*}
  \binom{n}{k}=\binom{n}{n-k}, \qquad\text{ for all $n\ge 0$ and all integers $k$,}
\end{equation*}
and the \textit{addition formula}
\begin{equation*}
  \binom{n}{k}=\binom{n-1}{k} + \binom{n-1}{k-1}, \qquad\text{ for all integers $n,k$,}
\end{equation*}
hold.

\numberwithin{theorem}{section}
\section{Characteristic $p>0$: Fooling-Set Matrices from Sequences}\label{sec:construction}
For a prime number~$p$, we denote by $\FF_p$ the finite field with~$p$ elements.  The following is the accurate statement of our result.

\begin{theorem}\label{thm:main:nz}
  For every prime number~$p$, there is a family of fooling-set matrices $\Mt$ over $\FF_p$ of size~$\nt$, $t=1,2,3,\dots$, such that $\nt \to \infty$, and
  \begin{equation*}
    \frac{ \nt }{ (\rk_{\FF_p} \Mt)^2 } \;\longrightarrow 1.
  \end{equation*}
\end{theorem}

As noted above, we use linear recurring sequences.  For every~$t$, we construct an $\nt$-periodic function, which gives us a fooling-set matrix of size~$\nt$.

\paragraph{\bf We now describe that construction.}  %
Let~$p$ be a prime number and $r \ge 2$ an integer.  Define the function $f\colon \ZZ\to \FF_p$ by the recurrence relation
\begin{subequations}\label{eq:def-f}
  \begin{equation}\label{eq:def-f:recrel}
    f(k+r) = -f(k) - f(k+1) \quad\text{for all $k\in \ZZ$}
  \end{equation}
  and the initial conditions
  \begin{equation}\label{eq:def-f:initial}
    f(0) = 1\text{, and } f(1) = \ldots = f(r-1) = 0.
  \end{equation}
\end{subequations}

Fix an integer $n > r$.  From the sequence, we define an $n\times n$ matrix as follows.  For ease of notation, the matrix indices are taken to be in $\{0,\dots,n-1\}\times \{0,\dots,n-1\}$.  We let
\begin{equation}\label{eq:def-M}
  M_{k,\ell} := f(k-\ell).
\end{equation}

It is fairly easy to see that $\rk M \le r$.

\begin{lemma}\label{lem:rk-M}
  The rank of~$M$ is at most~$r$.
\end{lemma}
\begin{proof}
  From~\eqref{eq:def-f:recrel}, for $k \ge r$, we deduce the equation $M_{k,\star} = -M_{k-r,\star} - M_{k-r+1,\star}$.
  Hence, each of the rows $M_{k,\star}$, $k \ge r$, is a linear combination of the first~$r$ rows of~$M$.
\end{proof}

It can be seen that the rank is, in fact, equal to~$r$: The top-left $r\times r$ submatrix is non-singular because it is upper-triangular with nonzeros along the diagonal.

\paragraph{\bf In the remainder of the section, we derive the fooling-set property.}
First, we reduce the fooling-set property~\eqref{eq:def-fool} of~$M$ to a property of the function~$f$.

\begin{lemma}\label{lem:fool-eq}\mbox{}
  The matrix~$M$ defined in~\eqref{eq:def-M} is a fooling-set matrix, if and only if,
  \begin{equation}\label{eq:cross-symmetry}
    f(k) f(-k) = 0  \quad\text{ for all $k \in \{1,\dots,n-1\}$.}
  \end{equation}
\end{lemma}
\begin{proof}
  It is clear from \eqref{eq:def-f:initial} and~\eqref{eq:def-M} that $M_{j,j} = f(0) = 1$ for all $j=0,\dots,n-1$, so it remains to verify~\eqref{eq:def-fool:off-diag}.
  Since
  \begin{equation*}
    M_{i,j} M_{j,i} = f(i-j) f(j-i) = f(i-j) f(-(i-j)),
  \end{equation*}
  if $f(k) f(-k) = 0$ for all $k=1,\dots,n-1$, then $M_{i,j} M_{j,i}$ is zero whenever $i\ne j$.  This proves~\eqref{eq:def-fool:off-diag}.
\end{proof}


Given appropriate conditions on $r$ and~$n$ (depending on~$p$), this condition on~$f$ can indeed be verified:

\begin{lemma}\label{lem:key-lemma}
  For all integers $t \ge 1$, if we let $r := p^t+1$ and $n := r(r-1)+1$, then $f(k)f(-k) = 0$ for all $k\in \ZZ\setminus n\ZZ$.
\end{lemma}

Combining the above three lemmas, we can complete the proof of Theorem~\ref{thm:main:nz}.

\begin{proof}[Proof of Theorem~\ref{thm:main:nz}.]
  Let~$p$ be a prime number.
  For every integer $t\ge 1$, let $r := p^t+1$ and $\nt := r(r-1) +1$, and define the matrix $\Mt := M$ over $\FF_p$ as in~\eqref{eq:def-M}.
  By Lemma~\ref{lem:rk-M}, the rank of $\Mt$ is at most~$r$, and from Lemmas \ref{lem:fool-eq} and~\ref{lem:key-lemma} we conclude that $\Mt$ is a fooling-set matrix.  Hence, we have
  \begin{equation*}
    1 \ge \frac{ \nt }{ \rk_{\FF_p} (\Mt)^2 } \ge \frac{ r^2-r+1 }{r^2} \ge 1 - p^{-t}/4 \xrightarrow{t\to\infty} 1,
  \end{equation*}
  where the left-most inequality is from~\eqref{eq:rk-of-fool}.
\end{proof}

To prove Lemma~\ref{lem:key-lemma}, we need two more lemmas.  The first one states that in every section $\{jr,\dots,(j+1)r-1\}$, $j=0,1,\dots$, there is a block of zeros whose length decreases with~$j$.

\begin{lemma}\label{lem:zero-blocks}
  For $j=0,\dots,r-2$, we have
  \begin{equation}\label{eq:zero-block}
    f(jr + i ) = 0 \quad\text{for $i = 1,\dots, r-1-j$.}
  \end{equation}
\end{lemma}
\begin{proof}
  Equation~\eqref{eq:zero-block} is true for $j=0$ by~\eqref{eq:def-f:initial}.
  Suppose~\eqref{eq:zero-block} holds for some $j<r-2$.  Then $f((j+1)r + i ) = 0$ for $i = 1,\dots, r-1-(j+1)$, because, by~\eqref{eq:def-f:recrel},
  \begin{equation*}
    f((j+1)r + i)
    =
    f(jr + i + r)
    = 
    -f(jr + i)  -  f(jr + (i+1))
    = -0 - 0
  \end{equation*}
  holds.
\end{proof}

Every function on~$\ZZ$ with values in a finite field which is defined by a (reversible) linear recurrence relation is periodic (cf.\ e.g.~\cite{LidlNiederreiter94}).  The second lemma establishes that a specific number~$n$ is a period of~$f$ as defined in~\eqref{eq:def-f}.

\begin{lemma}\label{lem:periodicity}
  If $r = p^t+1$ for some integer $t\ge1$, then $n := r(r-1) +1$ is a period of the function~$f$.
\end{lemma}
\begin{proof}
  In this proof, for convenience, we identify $\FF_p$ with the integers modulo~$p$.

  Consider $h(j,i) := f((j+1)r-i)$ for $i,j\in\ZZ$.  We have to show that
  \begin{subequations}\label{eq:periodicity:zZ}
    \begin{align}
      \label{eq:periodicity:zZ:0}
      h(r-1,0) &= 0.                              \\
      \label{eq:periodicity:zZ:middle}
      h(r-1,1) = \ldots = h(r-1,r-2) &= 0, \text{ and }\\
      \label{eq:periodicity:zZ:rminus1}
      h(r-1,r-1) &= 1.
    \end{align}
  \end{subequations}

  We will first prove the following claims.
  \begin{enumerate}[\it {Claim}~(a).]
  \item\label{claim:periodicity:rec} For all $i,j\in \ZZ$,
    \begin{equation*}
      h(j+1,i) = -h(j,i) -h(j,i-1).
    \end{equation*}
  \item \label{claim:periodicity:bd} For $j=0,\dots,r-3$
    \begin{equation*}
      h(j,-1) = 0, \ h(j,j+1) = 0.
    \end{equation*}
  \item\label{claim:periodicity:binom} For $j=0,\dots,r-2$ and $0\le i \le j$
    \begin{equation*}
      h(j,i) = (-1)^{j+1} \binom{j}{i} \mod p.
    \end{equation*}
  \end{enumerate}
  
  Before we prove the claims, we show how they imply~\eqref{eq:periodicity:zZ}.
  Recalling the well-known fact that
  \begin{equation*}
    \binom{p^t}{i} = 0 \mod p
  \end{equation*}
  for every integer~$t\ge 1$ and for all $i=1,\dots,p^t-1$ (cf.\ e.g.~\cite{LidlNiederreiter94}), the equations~\eqref{eq:periodicity:zZ:middle} follow by applying Claims \ref{claim:periodicity:rec} and~\ref{claim:periodicity:binom} with $j:=r-2$: For $i=1,\dots,r-2 = p^t-1$, since
  \begin{multline*}
    h(r-1,i) = -h(r-2,i) - h(r-2,i-1) =\\ = - (-1)^{r-1}\binom{r-2}{i} - (-1)^{r-1}\binom{r-2}{i-1} \mod p,
  \end{multline*}
  it follows that
  \begin{alignat*}{2}
    h(r-1,i)
    &= - \binom{r-1}{i} &\;&\mod p \\
    &= - \binom{p^t}{i} &&\mod p \\
    &= 0 &&\mod p.
  \end{alignat*}

  To prove~\eqref{eq:periodicity:zZ:rminus1}, we infer from the claims that
  \begin{multline*}
    h(r-1,r-1) = -h(r-2,r-1) -h(r-2,r-2) =
    \\
    - f( (r-1)r - r+1 ) - (-1)^{r-1} \binom{r-2}{r-2} =
    \\
    -f( (r-2)r + 1 ) -(-1)^{p} = 1,
  \end{multline*}
  where the last equation follows from Lemma~\ref{lem:zero-blocks} and the fact that $-(-1)^p=1$ even for $p=2$.
  Finally, for~\eqref{eq:periodicity:zZ:0}, we conclude that
  \begin{multline*}
    h(r-1,0) = -h(r-2,0) - h(r-2,-1) =
    \\
    -(-1)^{r-1} \binom{r-2}{0} - f( r^2 -(r-1) ) =
    -(-1)^p - h(r-1,r-1) =
    \\
    -(-1)^p -1 = 0,
  \end{multline*}
  where the last-but-one equation follows from \eqref{eq:periodicity:zZ:rminus1}.

  \smallskip%
  \subparagraph{\it Proof of Claim~(\ref{claim:periodicity:rec}).} 
  This is a straightforward computation.  For all $j,i$, we compute
  \begin{multline*}
    h(j+1,i) = f((j+2)r-i) =
    \\
    f((j+1)r-i+r) = -f((j+1)r-i) - f((j+1)r-(i-1)) =
    \\
    -h(j,i) -h(j,i-1).
  \end{multline*}
  \qed

  \smallskip%
  \subparagraph{\it Proof of Claim~(\ref{claim:periodicity:bd}).} 
  This claim follows from Lemma~\ref{lem:zero-blocks}.  We have
  \begin{align*}
    h(j,-1) &= f((j+1)r+1) = 0                        &&\text{for $j=0,\dots,r-3$,}\\
    \intertext{and} h(j,j+1) &= f((j+1)r-j-1) = f(jr+r-1-j) = 0 && \text{for $j=0,\dots,r-2$.}
  \end{align*}
  \qed
  
  \smallskip%
  \subparagraph{\it Proof of Claim~(\ref{claim:periodicity:binom}).} 
  Since $h(0,0)=-1$, Claim~(\ref{claim:periodicity:binom}), follows from Claims (\ref{claim:periodicity:rec}) and~(\ref{claim:periodicity:bd}).  \qed

  \medskip%
  \subparagraph{This completes the proof of Lemma~\ref{lem:periodicity}.}
\end{proof}

\begin{remark}
  As seen in the proof, not surprisingly, our recurrence relation~\eqref{eq:def-f:recrel} produces binomial coefficients.

  However, it would be interesting to know whether there are other linear recurrence relations, $f(k+r) = \sum_{j=0}^{r-1} \alpha_j f(k+j)$, which define circulant fooling-set matrices with the appropriate relation between size and rank.  Since all such sequences are periodic, only the conclusion of Lemma~\ref{lem:key-lemma} must be satisfied, and the period must be asymptotic to~$r^2$.
\end{remark}

Lemmas \ref{lem:zero-blocks} and~\ref{lem:periodicity} allow us to prove Lemma~\ref{lem:key-lemma}.

\begin{proof}[Proof of Lemma~\ref{lem:key-lemma}.]
  We need to show $f(k) f(-k) = 0$ whenever $n \nmid k$.  By Lemma~\ref{lem:periodicity}, this is equivalent to showing $f(k) f(n-k) = 0$ for $k=1,\dots,n-1$.  Given such a~$k$, let $j,i$ be such that $k = jr +i$ and $0\le i \le r-1$.
  
  If $i \le r-1-j$, then $f(k)=0$ by Lemma~\ref{lem:zero-blocks}, and we are done.  If, on the other hand, $i > r-1-j$, then
  \begin{equation*}
    n-k
    =
    r^2 - r +1 - jr - i
    =
    (r-1-(j+1))r  + (r-i+1),
  \end{equation*}
  and $r-i+1 \le j+1$, so, by Lemma~\ref{lem:zero-blocks}, we have $f(n-k) = 0$.
\end{proof}

\section{Characteristic Zero: Fooling-Set Matrices from Binomial Coefficients}\label{sec:chzero}
We now prove the result in characteristic zero.

\begin{theorem}\label{thm:main:0}
  For each~$r \ge 1$, there is a fooling-set matrix $\Mr$ over~$\QQ$ of size $\binom{r+1}{2}$ and rank~$r$.
\end{theorem}

The entries of~$\Mr$ are binomial coefficients, up to sign.  As in the previous section, the low rank property will follow from the binomial addition identity.  Whereas the matrix in the previous section is circulant, this matrix has a more complicated block structure but each block is Toeplitz.

\paragraph{\bf We now describe the construction of the matrices $\Mr$.}
To get some feeling for these matrices, here are the first few examples
\begin{equation*}
  \Mr[1]=\begin{pmatrix}
    1
  \end{pmatrix}\!,\ \Mr[2] = \begin{pmatrix}
    1 & 0 & 1\\
    -1 & 1 & 0 \\
    0 & 1 & 1\\
  \end{pmatrix}\!,\ \Mr[3] = \begin{pmatrix}
    1 & 0 & 0 & 1 & -1 & 1\\
    -1 & 1 & 0 & 0 & 1 & 0\\
    1 & -1 & 1 & -1 & 0 & 0\\
    0 & 1 & 0 & 1 & 0 & 1\\
    0 & 0 & 1 & -1 & 1 & 0\\
    0 & 1 & 1 & 0 & 1 & 1
  \end{pmatrix}.
\end{equation*}

The recursive structure of $\Mr$ can be seen from these examples\footnote{%
  If the reader wants to see larger examples, Matlab code to construct $\Mr$ can be found at \url{https://github.com/troyjlee/hadamard_factorization}.%
}.  %
In general, the top left $r\times r$ principal submatrix of $\Mr$ will be lower triangular with ones of alternating sign, and the bottom right $\binom{r}{2}$-sized principal submatrix will be $\Mr[r-1]$.

We now give the details of the construction.  First we define, for each integer~$t$, a function $f_t\colon \NN\times\NN \to \ZZ$ (with $\NN:= \{1,2,3,\dots\}$).  These functions will be used in the construction.  They can be thought of as infinite matrices, and we will use the notation $F_t^{r,s}$ to specify the $r\times s$ matrix
\begin{equation*}
  \bigl( F_t^{r,s} \bigr)_{i,j} := f_t(i,j), \quad\text{ for $i=1,\dots,r$ and $j=1,\dots,s$.}
\end{equation*}
Let $t \in \ZZ$ and $i,j \in \NN$.  The function $f_t$ is defined as
\begin{equation*}
  f_t(i,j) :=
  \begin{cases}
    \displaystyle \binom{t-1}{j-i-1}, & \text{ if $t>0$,}\\[2ex]
    \displaystyle (-1)^{j-i}\binom{-t-1+j-i}{-t-1}, & \text{ if $t \le 0$ and $i < j$,}\\[2ex]
    \displaystyle (-1)^{i-j-t} \binom{i-j-1}{-t}, & \text{ if $t \le 0$ and $i \ge j$.}
  \end{cases}
\end{equation*}
Note that in each case, $f_t(i,j)$ depends on the difference $i-j$ only, thus each $f_t$ is Toeplitz.  When $t>0$, we see that $f_t(i,j)=0$ whenever $i \ge j$ meaning that these~$f_t$ are upper triangular.  When $t=0$, the definition simplifies to $f_0(i,j)=\binom{-1}{i-j}$, thus $f_0$ is lower triangular with ones on the main diagonal.

To get a better idea where the $f_t$ come from, consider an extended Pascal's triangle where the upper and lower indices begin from $-1$.  In the following table, the entries are binomial coefficients where upper indices label the rows, lower indices label the columns.
\begin{center}
  \begin{tabular}{c|cccccc}
    &-1 &0&1&2&3&4 \\
    \hline
    -1 & 0 &1 & -1 & 1 & -1&1\\
    0 & 0 &1 & 0 & 0 & 0&0\\
    1 & 0 &1& 1 &0 & 0&0 \\
    2& 0 & 1 & 2 & 1& 0&0\\
    3& 0 & 1 &3 & 3&1&0 \\
    4& 0 & 1 & 4 & 6 & 4 & 1
  \end{tabular}
\end{center}
The matrix $f_t$ for $t >0$ is the infinite Toeplitz matrix whose first row is given by the row of Pascal's triangle indexed by $t-1$, and whose first column is all zero.  For $t < 0$, up to signs, $f_t$ is the infinite Toeplitz matrix whose first column is given by the column of Pascal's triangle indexed by $-t$ and whose first row is given by the $-t-1$ column of Pascal's triangle, starting from the row indexed by $-t-1$.
 
Using the $f_t$ we can now construct the fooling-set matrices~$\Mr$.  For $r \ge 1$, let $\Mr$ be a matrix of size $\binom{r+1}{2}$ defined as
\begin{equation*}
  \Mr =
  \begin{pmatrix}
    f_0^{r,r} & f_{-1}^{r,r-1} & f_{-2}^{r,r-2} & \cdots & f_{-r+1}^{r,1} \\[.75ex]
    f_1^{r-1,r} & f_0^{r-1,r-1} & f_{-1}^{r-1,r-2} &  & f_{-r+2}^{r-1,1} \\[.75ex]
    f_2^{r-2,r} & f_1^{r-2,r-1} & f_0^{r-2,r-2} &  & f_{-r+3}^{r-2,1} \\[1ex]
    \vdots &  &  & \ddots & \vdots \\[1ex]
    f_{r-1}^{1,r} & f_{r-2}^{1,r-1} &\cdots  & \cdots & f_0^{1,1} \\
  \end{pmatrix}
\end{equation*}
The size of $\Mr$ is clearly $\binom{r+1}{2}$.  That $\Mr$ is a fooling-set matrix and has rank~$r$ will be shown in the next lemmas.

\paragraph{\bf We first show that $\Mr$ is a fooling-set matrix.}
This follows from the fact that~$f_0$ is lower triangular and that in the above extended Pascal's triangle for $t > 0$ the row indexed by $t-1$ and column indexed by~$t$ are disjoint.

\begin{lemma}\label{lem:0:fool}
  $\Mr$ is a fooling-set matrix.
\end{lemma}
\begin{proof}
  The diagonal entries of $\Mr$ are~$1$ as desired.  To show that $\Mr(i,j)\Mr(j,i) = 0$ for $i \ne j$, it suffices to show that $f_t(i,j) f_{-t}(j,i)=0$ for each~$t$.  This clearly holds for $t=0$ as $f_0$ is lower triangular.  Now suppose $t >0$.  If $i \ge j$ then $f_t(i,j)=0$ thus in this case we are also fine.  In the case $j > i$ we have
  \begin{equation*}
    \abs{ f_t(i,j)} \abs{ f_{-t}(j,i) } = \binom{t-1}{j-i-1} \binom{j-i-1}{t}=0.
  \end{equation*}
  The second term is zero for $j-i \le t$ while the first term is zero for $j-i \ge t+1$, thus the product is always zero.
\end{proof}
In fact, $\Mr$ has the stronger property that exactly one of $\Mr(i,j), \Mr(j,i)$ is zero for $i \ne j$.

\paragraph{\bf We now come to the rank of~$\Mr$.}
The following claim is the key to prove $\rk(\Mr) \le r$.

\begin{lemma}\label{lem:recurrence}
  For any $t \in \ZZ$ and $i,j \in \NN$
  \begin{equation*}
    f_t(i,j)=f_{t-1}(i,j) + f_{t-1}(i+1,j).
  \end{equation*}
\end{lemma}
\begin{proof}
  We break the proof into three cases depending on the value of $t$.

  \subparagraph{Case 1: $t >1$} This case follows from the binomial addition formula
  \begin{align*}
    f_t(i,j)=\binom{t-1}{j-i-1}&=\binom{t-2}{j-i-1} + \binom{t-2}{j-i-2} \\
    &=f_{t-1}(i,j) + f_{t-1}(i+1,j) \enspace.
  \end{align*}

  \subparagraph{Case 2: $t=1$} In this case we use the symmetry identity together with binomial addition formula.
  \begin{align*}
    f_1(i,j)=\binom{0}{j-i-1} = \binom{0}{i-j+1}&=\binom{-1}{i-j} + \binom{-1}{i-j+1} \\
    & = f_0(i,j)+f_0(i+1,j) \enspace.
  \end{align*}

  \subparagraph{Case 3: $t \le 0$} First consider the case $i \ge j$.  Then again by the binomial addition formula
  \begin{align*}
    f_{t}(i,j)&=
    (-1)^{i-j-t}\binom{i-j-1}{-t} \\
    &=(-1)^{i-j-t} \Biggl(-\binom{i-j-1}{-t+1}+\binom{i-j}{-t+1} \Biggr) \\
    &=(-1)^{i-j-t+1}\binom{i-j-1}{-t+1}+(-1)^{i-j-t+2}\binom{i-j}{-t+1} \\
    &=f_{t-1}(i,j) + f_{t-1}(i+1,j) \enspace .
  \end{align*}

  Finally, consider the case $i<j$.  This case requires some care as it could be that $i+1=j$.  For $t<0$, however, notice that the two formulas defining $f_t$ agree when $i=j$.  The first gives $(-1)^{j-i}$ and the second $(-1)^{i-j-t} (-1)^{-t}=(-1)^{j-i}$.  Thus when $t<0$ and $i=j$ the two formulas in the definition are consistent.  As we are in Case 3, we are safe expressing $f_{t-1}(i+1,j)$ using the formula for $i < j$ as $t \le 0$.
  \begin{align*}
    f_t(i,j)&=(-1)^{j-i}\binom{-t-1+j-i}{-t-1} \\
    &=(-1)^{j-i} \Biggl( \binom{-t+j-i}{-t} - \binom{-t+j-i-1}{-t}  \Biggr) \\
    &=(-1)^{j-i} \binom{-t+j-i}{-t} + (-1)^{j-i-1}\binom{-t+j-i-1}{-t} \\
    &=f_{t-1}(i,j) + f_{t-1}(i+1,j) \enspace. \\
  \end{align*}
\end{proof}

\begin{lemma}\label{lem:0:rk}
  The rank of $\Mr$ is $r$.
\end{lemma}
\begin{proof}
  The rank of $\Mr$ is at least~$r$, because the submatrix $f_0^{r,r}$ has rank $r$.

  Lemma~\ref{lem:recurrence} shows that all rows of $\Mr$ can be expressed as linear combinations of the first~$r$ rows, thus also $\rk(\Mr) \le r$.
\end{proof}

\paragraph{\bf Putting it all together,} Theorem~\ref{thm:main:0} is obtained from Lemmas \ref{lem:0:fool} and~\ref{lem:0:rk}.

\section{Conclusion}\label{sec:conclusio}
We conclude by discussing some questions which remain open.

First of all, in characteristic zero, it would be interesting to know whether inequality~\eqref{eq:rk-of-fool} is asymptotically tight, or, more generally:
\begin{question}
  What is smallest constant~$C$ such that $n \le C\, (\rk_\kk M)^2$ for all $n\times n$ fooling-set matrices~$M$ over a field~$\kk$ of characteristic zero?
\end{question}
There is a possibility that, in characteristic zero, the minimum achievable rank on the right hand side of inequality~\eqref{eq:rk-of-fool} may depend not only on the characteristic, but on the field~$\kk$ itself.  Indeed, there are examples of zero-nonzero patterns for which the minimum rank of a matrix with that zero-nonzero pattern differs between $\kk = \QQ$ and $\kk = \RR$, see e.g.~\cite{KoppartyBhaskararao08}.

Secondly, while the construction in Section~\ref{sec:construction} for nonzero characteristic gives circulant matrices, the matrices in Section~\ref{sec:chzero} are not circulant.
\begin{question}
  Can the exponent on the rank in the inequality~\eqref{eq:rk-of-fool} be improved for {\em circulant} fooling-set matrices over $\kk$ with characteristic zero?
\end{question}


\end{document}